\documentclass{article}
\usepackage{amsmath,amsthm,amsfonts,amssymb,graphicx,epigraph,stmaryrd}
\usepackage{enumerate}
\usepackage{etoolbox}
\usepackage{color}
\usepackage{natbib}
\usepackage[colorlinks=true,citecolor=blue,pdfpagemode=UseNone,pdfstartview=FitH]{hyperref}

\allowdisplaybreaks[4]

\renewcommand{\d}{\,\mathrm{d}}
\newcommand{\dd}{\mathrm{d}}

\newcommand{\R}{\mathbb{R}}
\newcommand{\FFF}{\mathcal{F}}

\newcommand{\NNN}{\mathcal{N}}

\newcommand{\E}{\mathbb{E}}

\DeclareMathOperator{\UB}{\mathrm{UB}}
\DeclareMathOperator{\LB}{\mathrm{LB}}

\theoremstyle{plain}
\newtheorem{theorem}{Theorem}[section]

\newtheorem{proposition}[theorem]{Proposition}

\theoremstyle{definition}

\theoremstyle{remark}
\newtheorem{remark}[theorem]{Remark}

\title{The power of forgetting in statistical hypothesis testing}
\author{Vladimir Vovk}

\begin{document}
\maketitle

\epigraph{How can forgetfulness and efficiency coexist?
  Aren't these two concepts absolute opposites?
  Far from it.}
{Mike Byster, ``The Power of Forgetting'', 2014}

\begin{abstract}
  This paper places conformal testing
  in a general framework of statistical hypothesis testing.
  A standard approach to testing a composite null hypothesis $H$
  is to test each of its elements
  and to reject $H$ when each of its elements is rejected.
  It turns out that we can fully cover conformal testing using this approach
  only if we allow forgetting some of the data.
  However, we will see that the standard approach
  covers conformal testing in a weak asymptotic sense and under restrictive assumptions.
  I will also list several possible directions of further research,
  including developing a general scheme of online testing.

   The version of this paper at \url{http://alrw.net} (Working Paper 37)
   is updated most often.
   The conference version was published in COPA 2023
   (\emph{Proceedings of Machine Learning Research}, 204:347--366, 2023).
\end{abstract}

\section{Introduction}
\label{sec:introduction}

Conformal testing is an interesting application of conformal prediction.
It turns the p-values output by conformal prediction
into dynamic procedures for testing statistical hypotheses.
This paper is a high-level discussion of dynamic testing of statistical hypotheses,
in an attempt to place conformal testing
in the general theory of statistical hypothesis testing.

In conformal testing,
we are usually interested in testing the exchangeability model,
because of its importance in machine learning.
In this paper, however, we will often consider testing other statistical models,
which will shed new light on the relation of conformal testing
to alternative approaches.

Section~\ref{sec:batch_vs_online} briefly reviews
the history of the online approach to hypothesis testing.
Formal exposition starts in Sect.~\ref{sec:approaches},
which is a summary of three approaches to online testing,
including, in Sect.~\ref{subsec:conformal_testing}, a summary of conformal testing.
Section~\ref{sec:scheme} introduces a general scheme
covering all three approaches.
Section~\ref{sec:inadequacy} points out a seemingly unnatural feature of conformal testing,
which is forgetting some of the data
(see Remark~\ref{rem:unnatural} at the end of Sect.~\ref{sec:approaches}).
Section~\ref{sec:adequacy} explains that the extent to which we need forgetting is limited
(albeit under strong assumptions).
Section~\ref{sec:experiments} illustrates some of the points discussed in the earlier sections
using computer simulations,
and Sect.~\ref{sec:conclusion} concludes.

\section{From batch to online hypothesis testing}
\label{sec:batch_vs_online}

The classical theory of statistical hypothesis testing,
as created by Student (\citeyear{Gosset:1908}),
Fisher (\citeyear{Fisher:1925applications}),
Egon Pearson, and Neyman \citep{Neyman/Pearson:1933},
was developed in the batch setting (in the terminology of modern machine learning).
Given a batch of data $z_1,\dots,z_N$,
we would like to test the hypothesis (known as the \emph{null hypothesis})
that $z_1,\dots,z_N$ were generated from a given probability measure
(in which case the null hypothesis is called \emph{simple})
or a probability measure from a given family of probability measures
(in which case the null hypothesis is called \emph{composite}).
The number of observations $N$ (sample size) is chosen in advance.
The classical theory is still dominant in statistical hypothesis testing.

\begin{remark}
  {\upshape
  I do not list Karl Pearson because he was interested
  in statistical tests,
  such as his famous $\chi^2$ test \citep{Pearson:1900},
  that only have an asymptotic (``large-sample'') justification.
  The first exact (``small-sample'') test for an interesting composite null hypothesis
  was developed by Student (\citeyear{Gosset:1908}),
  whose results were rigorously proved and greatly developed
  by Fisher (\citeyear{Fisher:1925applications,Fisher:1925book,Fisher:1935book}).}
\end{remark}

The assumption that $N$ is chosen in advance was removed during World War II
by \citet{Wald:1945,Wald:1947} in the US,
with research along similar lines going on in the UK \citep{Barnard:1946}.
However, Wald's picture was not fully dynamic:
he just made $N$ a stopping time when the decision
(rejection or acceptance of the null hypothesis)
is announced.
The dynamic interpretation in which the likelihood ratio
is interpreted directly as the evidence in favour/against the null/alternative hypothesis
was given by \citet[pp.~459--460 and the last paragraph]{Barnard:1947}.
More recently,
this interpretation has been widely discussed under the name of the \emph{law of likelihood}
\citep{BF:2011}.
(The term ``law of likelihood'' was coined by \citealt[Chap.~5]{Hacking:1965},
but Hacking was only interested in its special case,
namely in comparison of the likelihood ratio with 1.)

The dynamic way of testing a simple null hypothesis has its origin
in Ville's (\citeyear{Ville:1939}) notion of a martingale.
The value of a test martingale
(i.e., a nonnegative martingale with initial value 1)
can be interpreted as the amount of evidence
found against the null hypothesis.
Ville did not have this interpretation in his book
(infinite sequences were his main object of interest),
but it formed gradually in the algorithmic theory of randomness;
e.g., it is stated explicitly in \citet{Vovk/Vyugin:1994}.
This interpretation is the basis of \citet{Shafer/Vovk:2001,Shafer/Vovk:2019}.
It is closely related to Barnard's \citeyear{Barnard:1947} paper mentioned earlier,
since a test martingale can often be represented in the form of a likelihood ratio.

\begin{remark}
  {\upshape
  In particular, for simple null hypotheses,
  a test martingale is a likelihood ratio.
  Therefore, it has a very convincing Bayesian interpretation:
  if \emph{a priori} we regard the null and the alternative
  (the numerator of the likelihood ratio)
  as equally probable, the posterior probability of the null will be $1/(L+1)$,
  where $L$ is the likelihood ratio.}
\end{remark}

How do test martingales work for composite hypotheses?
The standard way of testing a composite null hypothesis in the algorithmic theory of randomness
is to test against each element of the composite null
and then take the infimum of the resulting randomness deficiencies.
See, e.g., \citet{Vovk:1986-proofs}, \citet[Theorem~2]{Vovk/Vyugin:1993},
\citet[Sect.~4]{Bienvenu/etal:2011}, and \citet[Theorem~4.2.1]{Gacs:arXiv2105}.
This suggests gambling against all values of the parameter $\theta$
(indexing the null hypothesis)
obtaining a test martingale $S^{\theta}$ for each $\theta$
and then taking the infimum over $\theta$.
We will do this in Sect.~\ref{sec:inadequacy}.

\section{Three modern ways of dynamic hypothesis testing}
\label{sec:approaches}

In this section we will discuss three approaches, by now standard,
to dynamic hypothesis testing.
Only one of them, conformal testing,
can be, and has been, used for testing the exchangeability model
(the standard statistical model in machine learning)
in non-trivial cases.

First we introduce our framework and notation.
Let $(\Omega,\FFF)$ be a measurable space
equipped with a family $P_\theta$, $\theta\in\Theta$,
of probability measures on $(\Omega,\FFF)$.
We refer to $(\Omega,\FFF)$ as our \emph{sample space}
and to $(P_{\theta}\mid\theta\in\Theta)$ as our \emph{statistical model}.
We are not assuming that the model is parametric
(i.e., that $\Theta$ is a subset of a finite-dimensional Euclidean space $\R^n$);
e.g., $(P_{\theta}\mid\theta\in\Theta)$
may be the set of all exchangeable probability measures on $\R^\infty$.

Our random observations are $Z_1,Z_2,\dots$;
these are random elements on $(\Omega,\FFF)$ taking values
in a measurable space $\mathbf{Z}$,
which is our \emph{observation space}.
Let $z_1,z_2,\dots$ be the realizations of $Z_1,Z_2,\dots$.

Set $\FFF_n:=\sigma(Z_1,\dots,Z_n)$ for $n=0,1,\dots$,
i.e., $\FFF_n$ is the $\sigma$-algebra generated by the first $n$ observations.
The interpretation of $\FFF_n$ is the full information available by time $n$.
The sequence $(\FFF_n)$ of $\sigma$-algebras is called the \emph{natural filtration}.
In general, a \emph{filtration} is an increasing sequence of $\sigma$-algebras,
and we will often be interested in filtrations $(\FFF'_n)$
that are \emph{poorer} than the natural filtration $(\FFF_n)$
in the sense that $\FFF'_n\subset\FFF_n$ for some $n$
(typically, for all $n\ge1$).

\begin{remark}
  {\upshape
  It is more customary to start from a filtration $(\FFF_n)$
  and require that each $Z_n$ be measurable w.r.\ to $\FFF_n$ for each $n$.
  This is a more general setting allowing further sources of information
  apart from the observations $Z_1,Z_2,\dots$.
  We will, however, assume that the observations are the only source of information
  (and will even allow forgetting some aspects of the observations).}
\end{remark}

We regard $(P_{\theta}\mid\theta\in\Theta)$ as our null hypothesis,
and we would like to test whether $z_1,z_2,\dots$ were really generated
from one of the $P_\theta$.

If the statistical model contains only one probability measure $P$,
online testing consists in choosing a \emph{test martingale} $S_n$, $n=0,1,\dots$,
i.e., a sequence of random variables such that $S_n$ is $\FFF_n$-measurable,
$S_0=1$, and, for each $n=0,1,\dots$,
\[
  \E(S_{n+1}\mid\FFF_{n})
  =
  S_n.
\]
We regard $S_n$ as the capital at time $n$ of a gambler betting
against the null hypothesis $P$.
Next we will discuss three known ways of generalizing this definition
to composite null hypotheses.

\subsection{Element-wise testing}
\label{subsec:element-wise}

The most basic and standard generalization
is to gamble against each $P_\theta$ separately
and to regard the null hypothesis falsified
to the degree that all of $P_\theta$ have been falsified.
Formally, for each $\theta\in\Theta$,
we fix a test martingale $S^{\theta}$,
and we then define
\begin{equation}\label{eq:S}
  S_n
  :=
  \inf_{\theta\in\Theta}
  S^{\theta}_n.
\end{equation}
Any process $S$ that can be obtained in this way will be referred to
as an \emph{element-wise test},
and I will sometimes refer to this procedure of testing
as \emph{element-wise testing}.

\begin{remark}
  {\upshape
  The function $S^{\theta}_n(\omega)$ of $\theta\in\Theta$, $n$, and $\omega\in\Omega$
  is not assumed to be measurable in $\theta$,
  and so $S_n$ is not a random variable in general.
  (And even if $S^{\theta}_n(\omega)$ were assumed measurable in $\theta$,
  taking an infimum over an uncountable set may destroy measurability.)}
\end{remark}

A special case of element-wise testing \eqref{eq:S}
is used in \citet{Ramdas/etal:2022},
where each $S^\theta$ is defined as the likelihood ratio $\d Q/\d P_{\theta}$
and $Q$ is a probability measure that does not depend on $\theta$
(while dependence on $\theta$ is allowed in element-wise testing in general).
We will refer to this special case
as \emph{simple element-wise testing}.

\begin{remark}
  {\upshape
  \citet{Ramdas/etal:2022} apply their simple element-wise testing scheme
  to testing exchangeability,
  but, as we explain in \citet[Sect.~9.2.1]{Vovk/etal:2022book}
  (see, especially, Remarks~9.7 and~9.8),
  this scheme (based on the maximum likelihood estimate)
  is applicable to testing exchangeability only in toy situations.}
\end{remark}

\subsection{Pivotal testing}
\label{subsec:neo-fiducial}

The second approach goes back to Fisher's fiducial statistics
and was widely promoted by, e.g., George Barnard and Donald Fraser.
Recent work includes Peter McCullagh's (see \citealt{McCullagh/etal:2009})
and the work on confidence distributions,
including confidence predictive distributions
\citep{Schweder/Hjort:2016,Shen/etal:2018}.

An \emph{online pivotal model} is a pair $(N,Q)$,
where $N$ is a measurable mapping (\emph{normalizing transformation})
$N:\mathbf{Z}^*\to\mathbf{Z}'$ to some measurable space $\mathbf{Z}'$
and $Q$ is a probability measure on $(\mathbf{Z}')^{\infty}$.
We say that it \emph{agrees} with our statistical model $(P_\theta\mid\theta\in\Theta)$
if the distribution of the random sequence
\begin{equation}\label{eq:N}
  (Z'_1,Z'_2,\dots)
  :=
  (N(Z_1),N(Z_1,Z_2),N(Z_1,Z_2,Z_3),\dots),
\end{equation}
where $(Z_1,Z_2,\dots)\sim P_{\theta}$,
is $Q$ (in particular, it does not depend on the parameter $\theta$).
We might say that it \emph{strongly agrees} with our statistical model
if $\{P_\theta\mid\theta\in\Theta\}$ contains
all probability measures on $(\Omega,\FFF)$
for which \eqref{eq:N} is distributed as $Q$;
however, we will not use this stronger notion.

\begin{remark}
  {\upshape
  The definition of an online pivotal model
  can be trivially extended by allowing $N$ to depend on the parameter value $\theta\in\Theta$.
  Such an extension would even better agree with the term ``pivot'',
  since in statistics pivotal quantities are allowed to depend on $\theta$
  (those that do not depend on $\theta$ are usually called ``ancillary statistics'',
  but a disadvantage of the term ``ancillary statistic''
  is that it is usually associated with conditional inference).
  In this paper we will only be interested in examples
  where the normalizing transformation $N$ does not depend on $\theta$.}
\end{remark}

Let me give three simple examples of online pivotal models.
The \emph{full Gaussian pivotal model} is $(N,Q)$ where
\begin{align*}
  N(z_1)
  &:=0,\\
  N(z_1,\dots,z_n)
  &:=
  (z_n - z_1)/(z_2-z_1)
  \quad\text{for $n\ge2$},
\end{align*}
and $Q$ is the push-forward of the standard Gaussian measure $\NNN^{\infty}_{0,1}$ on $\R^{\infty}$
under the mapping $(Z_1,Z_2,\dots)\mapsto(Z'_1,Z'_2,\dots)$ defined by \eqref{eq:N}.
(Let us set, e.g., $0/0:=0$, here and below.)
This online pivotal model agrees
with the 2-parameter Gaussian statistical model
$(\NNN^{\infty}_{\mu,\sigma^2}\mid\mu\in\R,\;\sigma>0)$
(where $\NNN_{\mu,\sigma^2}$ is parametrized by the mean $\mu$ and variance $\sigma^2$).
The other two example are, in some sense, submodels of this model.

The \emph{Gaussian pivotal model with variance 1} is $(N,Q)$ where
\begin{equation}\label{eq:N-var1}
  N(z_1,\dots,z_n)
  :=
  z_n - z_1
  \quad\text{for $n\ge1$},
\end{equation}
and $Q$ is the push-forward of $\NNN^{\infty}_{0,1}$ under \eqref{eq:N}.
This online pivotal model agrees
with the 1-parameter Gaussian statistical model
$(\NNN^{\infty}_{\mu,1}\mid\mu\in\R)$ with the variance fixed to 1.
The \emph{Gaussian pivotal model with mean 0} is $(N,Q)$ where
\begin{equation*}
  N(z_1,\dots,z_n)
  :=
  z_n / z_1
  \quad\text{for $n\ge1$},
\end{equation*}
and $Q$ is the push-forward of $\NNN^{\infty}_{0,1}$ under \eqref{eq:N}.
This online pivotal model agrees
with another 1-parameter Gaussian statistical model,
$(\NNN^{\infty}_{0,\sigma^2}\mid\sigma>0)$,
with the mean fixed to 0.

For further examples,
see \citet{McCullagh/etal:2009} (Gauss linear model) and \citet[Sect.~4.1]{Ramdas/etal:2023}.

An online pivotal model reduces (perhaps not perfectly)
a composite null hypothesis to a simple one,
and gambling against a simple null hypothesis is unproblematic.
Formally, set
\begin{equation}\label{eq:sigma-pivot}
  \FFF'_n:=\sigma(Z'_1,\dots,Z'_n)
\end{equation}
for $n=0,1,\dots$,
so that the filtration $(\FFF'_n)$ is typically poorer than the natural filtration $(\FFF_n)$.
Choose a test martingale $S$ w.r.\ to the filtration $(\FFF'_n)$ and probability measure $Q$.
We will then refer to $S$ as a \emph{pivotal test martingale}.

Standard uses of online pivotal models are for producing
prediction sets \citep{McCullagh/etal:2009}, confidence predictive distributions
(\citealp[Sect.~12.4]{Schweder/Hjort:2016};
\citealp{Shen/etal:2018}),
and confidence distributions
(\citealp{Cox:1958}; \citealp{Xie/Singh:2013}; \citealp{Schweder/Hjort:2016}).
However, their adaptation to testing is straightforward,
and is analogous to the step from conformal prediction to conformal testing.

\subsection{Conformal testing}
\label{subsec:conformal_testing}

The exposition in this paper is intended to be self-contained
(apart from the definition of Bayes--Kelly test martingales in Sect.~\ref{sec:experiments}),
but for further details about online compression models,
see \citet[Part~IV]{Vovk/etal:2022book}.

An \emph{online compression model} is a quadruple $(\Sigma,\Box,F,B)$,
where
\begin{itemize}
\item
  $\Sigma$ is a measurable space,
  which is called the \emph{summary space}
  and whose elements are called \emph{summaries};
\item
  $\Box\in\Sigma$ is a fixed summary called the \emph{empty summary};
\item
  $F:\Sigma\times\mathbf{Z}\to\Sigma$ is a measurable function
  called the \emph{forward function};
\item
  $B$ is a Markov kernel mapping $\Sigma$
  to the probability measures on $\Sigma\times\mathbf{Z}$
  such that
  \[
    B(F^{-1}(\sigma)\mid\sigma)=1
  \]
  for each $\sigma\in F(\Sigma\times\mathbf{Z})$.
\end{itemize}
An alternative,
often more convenient
(especially for defining specific examples)
representation of online compression models
is in terms of the corresponding repetitive structures.
Namely, the \emph{repetitive structure}
corresponding to an online compression model $(\Sigma,\Box,F,B)$
consists of the \emph{summarising statistic} $t:\mathbf{Z}^*\to\Sigma$
defined by
\begin{align*}
  t() &:= \Box,\\
  t(z_1,\dots,z_n) &:= F(t(z_1,\dots,z_{n-1}),z_n)
  \quad
  n=1,2,\dots,
\end{align*}
and the inverse transformation mapping
each $\sigma\in t(\mathbf{Z}^n)$ for each $n\in\{1,2,\dots\}$
to the probability measure $P_n(\sigma)$ on $\mathbf{Z}^n$
defined by
\begin{multline*}
  P_n(\dd z_1,\dots,\dd z_n\mid\sigma_n)
  :=
  B(\dd\sigma_0,\dd z_1\mid\sigma_1)
  B(\dd\sigma_1,\dd z_2\mid\sigma_2)
  \dots\\
  B(\dd\sigma_{n-2},\dd z_{n-1}\mid\sigma_{n-1})
  B(\dd\sigma_{n-1},\dd z_n\mid\sigma_n).
\end{multline*}
We say that a probability measure $P$ on the sample space \emph{agrees}
with the online compression model
if, for each $n$, $P_n$ is a version of the conditional probability,
under $P$,
of the first $n$ observations given their summary.
And we say that a statistical model $(P_\theta\mid\theta\in\Theta)$ \emph{agrees}
with the online compression model if each $P_\theta$ does.
(As in the case of online pivotal models,
we do not require that $(P_\theta\mid\theta\in\Theta)$ contain
every probability measure that agrees with the online compression model.)

A \emph{conformity measure} in an online compression model
$(\Sigma,\Box,F,B)$
is a measurable function $A:\Sigma\times\mathbf{Z}\to\overline{\R}$.
The \emph{p-value} generated by the corresponding conformal predictor
after observing $(z_1,\dots,z_n)\in\mathbf{Z}^n$ is
\begin{multline}\label{eq:p}
  p_n
  :=
  B_{\mathbf{Z}}
  \left(
    \left\{
      z\in\mathbf{Z}
      \mid
      A(\sigma_n,z)
      <
      A(\sigma_n,z_n)
    \right\}
    \mid
    \sigma_n
  \right)\\
  +
  \tau_n
  B_{\mathbf{Z}}
  \left(
    \left\{
      z\in\mathbf{Z}
      \mid
      A(\sigma_n,z)
      =
      A(\sigma_n,z_n)
    \right\}
    \mid
    \sigma_n
  \right),
\end{multline}
where $B_{\mathbf{Z}}$ is the marginal distribution
\[
  B_{\mathbf{Z}}(E\mid\sigma)
  :=
  B(\Sigma\times E\mid\sigma)
\]
and $\tau_n\in[0,1]$
(in applications, $\tau_n$ is a number produced by a random number generator).

The main property of validity of conformal prediction
is that the p-values $p_1,p_2,\dots$ output according to \eqref{eq:p}
are independent and distributed uniformly on $[0,1]$
provided the observations are generated from a probability measure
that agrees with the online compression model
and the random numbers $\tau_1,\tau_2,\dots$ are distributed uniformly on $[0,1]$
and independent of the observations and between themselves.
Let $\FFF'_n$ be the $\sigma$-algebra generated by $p_1,\dots,p_n$,
\begin{equation}\label{eq:sigma-OCM}
  \FFF'_n:=\sigma(p_1,\dots,p_n).
\end{equation}
A \emph{conformal test martingale} is a test martingale
w.r.\ to the filtration $(\FFF'_n)$
and the uniform probability measure on $(p_1,p_2,\dots)\in[0,1]^\infty$
(the latter determining the probability measure on $\sigma(\cup_n\FFF'_n)$
underlying the martingale).

Now we can give four standard examples of online compression models,
which we do in terms of the corresponding repetitive structures.
The \emph{exchangeability model} has
$\sigma_n=\lbag z_1,\dots,z_n\rbag$ as the summary of a data sequence $(z_1,\dots,z_n)$,
and $P_n(\sigma_n)$ is the uniform distribution on all orderings of $\sigma_n$
(a fuller definition,
dealing carefully with the possibility of repetitions among the elements of $\sigma_n$,
is that $P_n(\sigma_n)$ is the push-forward of the uniform probability measure
on the $n!$ permutations $\pi:\{1,\dots,n\}\to\{1,\dots,n\}$
under the mapping $\pi\mapsto(z_{\pi(1)},\dots,z_{\pi(n)})$).
In fact, in all but one of our four examples,
$P_n(\sigma_n)$ will be the uniform probability measure on $t_n^{-1}(\sigma_n)$,
where $t_n$ is the restriction of $t$ to $\mathbf{Z}^n$.

In the remaining three examples,
the observation space is the real numbers, $\mathbf{Z}:=\R$.
The \emph{full Gaussian compression model} has the summarizing statistic
\[
  \sigma_n
  =
  t_n(z_1,\dots,z_n)
  :=
  \left(
    \sum_{i=1}^n z_i,
    \sum_{i=1}^n z^2_i
  \right)
\]
(equivalently, the summary of a data sequence $z_1,\dots,z_n$
consists of its empirical mean and standard deviation).
The summarizing statistic for the \emph{Gaussian compression model with variance 1} is
\begin{equation}\label{eq:var_1}
  \sigma_n
  =
  t_n(z_1,\dots,z_n)
  :=
  \sum_{i=1}^n z_i,
\end{equation}
and for the \emph{Gaussian compression model with mean 0} it is
\[
  \sigma_n
  =
  t_n(z_1,\dots,z_n)
  :=
  \sum_{i=1}^n z^2_i.
\]
The conditional distribution $P_n(\sigma_n)$ on $t_n^{-1}(\sigma_n)$
is defined to be the uniform distribution
in the case of the full Gaussian compression model
and Gaussian compression model with mean 0;
in both of these cases $t_n^{-1}(\sigma_n)$ is a sphere,
and the notion of the uniform distribution is meaningful and unambiguous.
For the Gaussian model with variance 1,
whose summarizing statistic is given by \eqref{eq:var_1},
$t_n^{-1}(\sigma_n)$ is not compact for $n>1$,
and the uniform distribution on it does not even exist;
we define $P_n(\sigma_n)$ as the probability measure on $t_n^{-1}(\sigma_n)$
with density proportional to
\begin{equation}\label{eq:density}
  \exp
  \left(
    -\frac12
    \sum_{i=1}^n z_i^2
  \right).
\end{equation}

\begin{remark}
  {\upshape
  The full Gaussian compression model
  (usually referred to simply as the Gaussian compression model)
  is the most general of our three Gaussian compression models,
  but it is easy to extend to a standard model of linear regression,
  the Gauss linear model,
  both in the pivotal \citep{McCullagh/etal:2009}
  and conformal \citep[Sect.~11.4.2]{Vovk/etal:2022book} cases.}
\end{remark}

It is interesting that the p-values $p_3,p_4,\dots$
output by any of these three Gaussian compression models are almost surely deterministic
(do not depend on the random numbers $\tau$),
while $p_1$ has the uniform distribution on $[0,1]$.
The second p-value $p_2$ behaves like $p_1$ in the case of the full Gaussian compression model
and like $p_3,p_4,\dots$ for the other two models.

Each Gaussian probability measure $\NNN^{\infty}_{\mu,\sigma^2}$
agrees with the full Gaussian compression model,
each $\NNN^{\infty}_{\mu,1}$ agrees with the Gaussian compression model with variance 1,
and each $\NNN^{\infty}_{0,\sigma^2}$ agrees with the Gaussian compression model with mean 0.
The density \eqref{eq:density} and the uniform density in the other two cases
can be obtained from this agreement.

The example of the Gaussian model with variance 1
will be most useful for us in this paper (see Sect.~\ref{sec:inadequacy} below).
In the case of pivotal models it is clearly the simplest one
among those that we discussed.
In the case of online compression modelling,
the summarizing statistic \eqref{eq:var_1} is also the simplest one,
but, unusually, the conditional distributions $P_n(\sigma_n)$ are not uniform
(another such example is discussed in \citealt[Sect.~11.3.7]{Vovk/etal:2022book}).

Conformal testing is able to produce non-trivial conformal test martingales
under the standard assumption of exchangeability for two reasons
\cite[Sect.~8.6.1]{Vovk/etal:2022book}:
\begin{itemize}
\item
  these conformal test martingales use a filtration that is poorer
  than the natural filtration generated by the observations $Z_1,Z_2,\dots$
  (we are forgetting some information);
\item
  the martingales are \emph{randomized},
  in the sense of depending on the random numbers $\tau_n$.
\end{itemize}
In this paper I will concentrate on the first reason
(which appears to be more important).
To get rid of the second reason,
we will consider an online compression model that does not require it
(in the sense that the p-values do not depend on the $\tau$s
apart from the first one,~$p_1$).

\begin{remark}\label{rem:unnatural}
  {\upshape
  The expression ``unnatural feature'' used in Sect.~\ref{sec:introduction}
  referred to the underlying filtration $(\FFF'_n)$ defined by~\eqref{eq:sigma-OCM}
  being different from the natural filtration.
  (Of course, the filtration $(\FFF'_n)$ defined by~\eqref{eq:sigma-pivot}
  also involves forgetting.)}
\end{remark}

\section{General scheme of online testing}
\label{sec:scheme}

In this section we sketch (somewhat informally) a general testing scheme
covering conformal testing and the other two approaches considered
in the previous section.
When processing the random observations $Z_1,Z_2,\dots$
while testing $(P_\theta\mid\theta\in\Theta)$ as our null hypothesis,
we proceed as follows.
\begin{enumerate}
\item\label{it:random}
  We use a random number generator producing independent $\tau_1,\tau_2,\dots$
  that are uniformly distributed on $[0,1]$;
  the sequence of $\tau$ is required to be independent of $Z_1,Z_2,\dots$
  for each $\theta\in\Theta$.
\item\label{it:p}
  We then transform the sequence of observations $z_1,z_2,\dots$
  and each parameter value $\theta\in\Theta$ to $z^{\theta}_1,z^{\theta}_2,\dots$:
  each $z^{\theta}_n$ is a function of $\theta$ and $z_1,\dots,z_n$
  that is measurable for each fixed value of $\theta$
  (no measurability in $\theta$ is required).
  Typically this step reduces the information
  contained in $z_1,z_2,\dots$ (for each $\theta$).
\item\label{it:gamble}
  Next, for each $\theta\in\Theta$, we gamble
  against the reduced sequence $z^{\theta}_1,z^{\theta}_2,\dots$
  and $\tau_1,\tau_2,\dots$
  obtaining a test martingale $S^\theta$
  (w.r.\ to the reduced filtration extended by the $\tau$)
  under $P_\theta$.
  Equivalently, we gamble against the extended observations $(z^{\theta}_n,\tau_n)$.
  Our capital $S^\theta_n$ at time $n$
  is a function of $(z^\theta_1,\tau_1),\dots,(z^\theta_n,\tau_n)$.
\item\label{it:inf}
  Finally, we use \eqref{eq:S} as the amount of evidence
  found against $(P_\theta)$ at time $n$.
\end{enumerate}

A natural question is whether this scheme is really general,
but it does cover the three methods described in the previous section.
These are special cases:
\begin{itemize}
\item
  In the simple element-wise testing scheme of \citet{Ramdas/etal:2022}
  (see Sect.~\ref{subsec:element-wise}),
  item~\ref{it:random} is not needed,
  the transformation in item~\ref{it:p} is identical
  (i.e., there is no transformation),
  and the gambling method in item~\ref{it:gamble}
  is to use likelihood ratios with $P_\theta$ in the denominator
  and the same probability measure (a mixture of $P_\theta$s)
  in the numerator.
\item
  In the neo-fiducial testing of Sect~\ref{subsec:neo-fiducial},
  item~\ref{it:random} is not needed.
  The transformation in item~\ref{it:p}
  and gambling in item~\ref{it:gamble}
  do not depend on $\theta$,
  and item~\ref{it:inf} is not needed.
\item
  In conformal testing,
  the transformation in item~\ref{it:p}
  and gambling in item~\ref{it:gamble}
  do not depend on $\theta$.
  Therefore, item~\ref{it:inf} is not needed.
  For some online compression models,
  such as the Gaussian models discussed earlier,
  item~\ref{it:random} is also not needed
  (apart from the first few p-values).
\end{itemize}

\section{Need for forgetting}
\label{sec:inadequacy}

Conformal testing often works well for testing the exchangeability model
\citep[Part~III]{Vovk/etal:2022book}.
On the other hand, it is obvious that,
without forgetting,
no successful gambling is possible against the null hypothesis of exchangeability,
or even against the stronger model of randomness \citep[Sect.~2.1.1]{Vovk/etal:2022book}:
if under the null hypothesis
there are no restrictions on the probability distribution of one observation,
our capital can only go down (or stay at the same level).
This is discussed in detail in \citet[Sect.~8.6.1]{Vovk/etal:2022book}
and stated in \citet{Ramdas/etal:2022} as Theorem~17.

Therefore, it is essential to allow the test martingales in the element-wise scheme
to depend on the value of the parameter $\theta$ if we want to avoid forgetting.
In this section we will give an example
where even such dependence does not allow us to recover
results attainable by pivotal and conformal methods.

As before,
we observe a sequence $Z_1,Z_2,\dots\in\mathbf{Z}$
generated by a probability measure in a family $(P_\theta\mid\theta\in\Theta)$,
and we would like to have an online measure of evidence found against $(P_\theta\mid\theta\in\Theta)$
as null hypothesis.
For each $\theta$,
we take a test martingale $S^\theta$ w.r.\ to $P_\theta$
and the natural filtration $\FFF=(\FFF_n)$ (i.e., $\FFF_n$ is generated by $Z_1,\dots,Z_n$),
and consider the element-wise test \eqref{eq:S}
as the amount of evidence
found against $(P_\theta\mid\theta\in\Theta)$ at time $n$.

\subsection{An example for pivotal testing}
\label{subsec:pivotal}

The following simple example shows the inadequacy of element-wise tests.
We are testing the Gaussian pivotal model with variance 1,
or the statistical model $(\NNN^{\infty}_{\mu,1}\mid\mu\in\R)$.
The normalizing transformation \eqref{eq:N-var1}
acts on the random observations as
\[
  Z_1,Z_2,\ldots
  \mapsto
  Z'_1,Z'_2,\dots,
\]
where $Z'_n:=Z_n-Z_1$,
so that $Z'_n\sim\NNN_{0,2}$ for $n\ge2$.
Consider the process
\[
  S_n
  :=
  \begin{cases}
    1 & \text{if $n\le1$} \\
    1/\NNN_{0,2}([-1,1]) & \text{if $n\ge2$ and $Z'_2\in[-1,1]$} \\
    0 & \text{if $n\ge2$ and $Z'_2\notin[-1,1]$}.
  \end{cases}
\]
It can be considered both as a function of $Z_1,Z_2,\dots$
and as a function of $Z'_1,Z'_2,\dots$,
but it is a martingale only as a function of $Z'_1,Z'_2,\dots$
(i.e., w.r.\ to the reduced filtration $(\FFF'_n)$,
where $\FFF'_n$ is generated by $Z'_1,\dots,Z'_n$).
If we express it as a function of $Z_1,Z_2,\dots$, it becomes
\begin{equation}\label{eq:S_n}
  S_n
  =
  \begin{cases}
    1 & \text{if $n\le1$} \\
    1/\NNN_{0,2}([-1,1]) & \text{if $n\ge2$ and $Z_2-Z_1\in[-1,1]$} \\
    0 & \text{if $n\ge2$ and $Z_2-Z_1\notin[-1,1]$}.
  \end{cases}
\end{equation}

Let us check that $S$ is not covered by element-wise testing,
i.e., $S_n\le\inf_{\mu}S^{\mu}_n$ is violated for some $n$,
$S^{\mu}$ being a natural test martingale
(i.e., a test martingale w.r.\ to the natural filtration)
under $\NNN^{\infty}_{\mu,1}$.
In fact, we will see that $S_n$ is not dominated
by any natural test martingale $S_n^{\mu}$ at times $n=1$ and $n=2$.
Indeed, if it is, we must have
\[
  S_1^{\mu}(z_1)
  \ge
  \max
  \left(
    \frac{\NNN_{\mu,1}([z_1-1,z_1+1])}{\NNN_{0,2}([-1,1])},
    1
  \right)
\]
for any $z_1\in\R$.
Notice that
\[
  \frac{\NNN_{\mu,1}([z_1-1,z_1+1])}{\NNN_{0,2}([-1,1])} > 1
\]
holds for a nontrivial range of $z_1$:
for example,
\[
  \frac{\NNN_{\mu,1}([\mu-1,\mu+1])}{\NNN_{0,2}([-1,1])}
  =
  \frac{\NNN_{0,1}([-1,1])}{\NNN_{0,2}([-1,1])}
  \approx
  1.31
  >
  1.
\]
Therefore, the expectation of $S_1^{\mu}(z_1)$ under $z_1\sim\NNN_{\mu,1}$ must exceed 1,
which contradicts $S^\mu$ being a test martingale.

\subsection{An example for conformal testing}
\label{subsec:conformal}

Our example for conformal testing will be a simple modification
of the example for pivotal models given in the previous subsection.
We again consider the Gaussian model with variance 1,
but now it is the online compression model
with summarizing statistic \eqref{eq:var_1} and conditional density \eqref{eq:density}.
The equality
\[
  z_1^2 + z_2^2
  =
  \frac{(z_1+z_2)^2}{2}
  +
  \frac{(z_1-z_2)^2}{2}
\]
shows that $P_2(\sigma_2)$ generates $(Z_1,Z_2)$ with $Z_2-Z_1\sim\NNN_{0,2}$,
and so
\begin{equation}\label{eq:p_2}
  p_2
  =
  \Phi((Z_2-Z_1)/\sqrt{2})
\end{equation}
($\Phi$ standing for the standard Gaussian cumulative distribution function)
if we choose $A(\sigma,z):=z$ as conformity measure.

Now we have a conformal test martingale
\begin{equation*}
  S_n
  :=
  \begin{cases}
    1 & \text{if $n\le1$} \\
    1/\NNN_{0,2}([-1,1]) & \text{if $n\ge2$ and $p_2\in[\Phi(-1/\sqrt{2}),\Phi(1/\sqrt{2})]$} \\
    0 & \text{if $n\ge2$ and $p_2\notin[\Phi(-1/\sqrt{2}),\Phi(1/\sqrt{2})]$}
  \end{cases}
\end{equation*}
in analogy with \eqref{eq:S_n};
in fact, as function of $Z_1,Z_2,\dots$ it is identical to \eqref{eq:S_n}.
As before, it is not dominated by any natural martingale $S^{\mu}$ w.r.\ to any $\NNN^{\infty}_{\mu,1}$.

\subsection{Another way of forgetting}

The main concept of forgetting studied in this paper
is replacing the natural filtration $(\FFF_n)$
by poorer filtration $(\FFF'_n)$.
In the case of the pivotal models,
this is the only natural kind of forgetting.
For the Gaussian model with variance 1,
we forget the first observation $z_1$,
and it leads to forgetting in the sense
of reducing the filtration, \eqref{eq:sigma-pivot}.

However, in the case of online compression models,
the very definition of such models is sometimes explained in terms of forgetting:
the summary $t(z_1,\dots,z_n)$ represents all useful information
contained in the data sequence $z_1,\dots,z_n$;
the rest is noise and can be forgotten.
In statistical terminology, the summary is a sufficient statistic.
Let me call this ``forgetting${}_2$''.
This is very different from forgetting (``forgetting${}_1$'') in the sense of reducing the filtration,
\eqref{eq:sigma-OCM}.
(This terminology is used only in this subsection;
elsewhere, ``forgetting'' always means ``forgetting${}_1$''.)

The main difference between forgetting${}_1$ and forgetting${}_2$
is that the latter is justified when we fully trust the model,
while the former is used when we are testing the model.
The information in $z_1,\dots,z_n$ that is not contained already in the summary
is noise only under the model,
but has great diagnostic value for testing the model;
on the other hand, the summary is of limited use for testing
(or even completely useless,
as when testing exchangeability for a finite time horizon).

\section{Element-wise testing partially works for a fixed horizon}
\label{sec:adequacy}

In this section we give theoretical results showing that the power of forgetting is limited,
unfortunately in a very weak sense.

\subsection{Finite horizon}

We start from a simple result for a finite \emph{horizon} $N$
(i.e., we have only $N$ observations,
or are only interested in the first $N$ observations).

\begin{proposition}\label{prop:exact}
  Let $N\in\{1,2,\dots\}$,
  and let $(S^\theta)$ be a family of test martingales
  w.r.\ to the same filtration (perhaps not natural)
  and a statistical model $(P_{\theta})$.
  Then there exists a family of natural test martingales $(\tilde S^\theta)$
  such that
  \[
    \inf_{\theta\in\Theta} \tilde S^\theta_N
    =
    \inf_{\theta\in\Theta} S^\theta_N.
  \]
\end{proposition}

Of course, $(S^\theta)$ being a family of test martingales
w.r.\ to $(P_{\theta})$ means that $S^\theta$ is a test martingale w.r.\ to $P_{\theta}$
for each $\theta$.
For example, $(S^\theta)$ may be a family of test martingales
from the scheme of Sect.~\ref{sec:scheme}.
The natural element-wise test $\inf_{\theta}\tilde S^{\theta}_n$
in Proposition~\ref{prop:exact}
can sometimes be less than the original element-wise test $\inf_{\theta}S^{\theta}_n$
at some time $n<N$,
but it will eventually catch up
(always, not just almost surely).

\begin{proof}[Proof of Proposition~\ref{prop:exact}]
  Let us fix a family of test martingales $(S^\theta)$.
  The expectation of our capital $S^\theta_N$ at step $N$ is 1,
  and for each $\theta$ we get a natural test martingale $\tilde S^{\theta}_n$,
  $n\in\{0,1,\dots,N\}$,
  by setting $\tilde S^\theta_N := S^\theta_N$ and averaging backwards:
  \begin{equation}\label{eq:main}
    \tilde S^{\theta}_n
    :=
    \E^\theta(\tilde S^{\theta}_{n+1}\mid\FFF_n),
    \quad
    n=N-1,\dots,0,
  \end{equation}
  where $\E^\theta$ stands for the expectation w.r.\ to $P_\theta$.
\end{proof}

Suppose we are given a test martingale $S$ that is not natural,
such as the ones used in our examples
in Sect.~\ref{subsec:pivotal} and Sect.~\ref{subsec:conformal}.
A disadvantage of Proposition~\ref{prop:exact} is that for steps before $N$
the backward averaging \eqref{eq:main} may give a result different from
(and therefore not dominating)
$S_n$, $n<N$.
Another disadvantage of Proposition~\ref{prop:exact}
is that it ignores the complexity,
in any sense (computational, descriptional, etc.),
of the natural test martingale $\tilde S$.
While $S$ may be very easy to define and independent of $\theta$,
such as a Composite Jumper conformal test martingale
\citep{Vovk/etal:2022book},
$\tilde S$ will depend on $\theta$ and may be much more complicated.

The right-hand side of the definition \eqref{eq:S} of element-wise testing
involves $\inf_{\theta}$ and so does not even have to be measurable,
as we already mentioned,
and in some sense it is not even well-defined when $\Theta$ is uncountable:
for each $\theta\in\Theta$ the corresponding test martingale is defined
to within a $P_{\theta}$-null set,
which makes the definition of $\inf_{\theta}$ non-invariant
w.r.\ to the choice of versions of conditional distributions.
(For rich spaces $\Theta$ and $\Omega$,
we can even make $\inf_{\theta}\tilde S^{\theta}_n=0$, $n<N$,
by an awkward choice of versions of the conditional expectations
in \eqref{eq:main}.)

The idea in the proof of Proposition~\ref{prop:exact}
can also be applied to randomized test martingales
(such as conformal test martingales under the exchangeability model).
Suppose $S_n=S_n^\theta$ does not depend on $\theta$
(as conformal test martingales).
We can then average $S_N$ w.r.\ to the random numbers $\tau_1,\dots,\tau_N$
and after that apply averaging w.r.\ to the $\sigma$-algebras $\FFF_n$:
\begin{equation}\label{eq:mean_BK}
  \tilde S_N^\theta
  :=
  \E^{\tau}(S_N),
  \qquad
  \tilde S^\theta_n
  :=
  \E^\theta
  \left(
    \tilde S_N^\theta \mid \FFF_n
  \right),
  \quad
  n=N-1,\dots,0,
\end{equation}
where, of course, $\E^{\tau}$ refers to averaging over the random numbers
(produced independently from the uniform distribution on $[0,1]$);
of course, there is no actual dependence of $\tilde S_n^\theta$ on $\theta$ for $n:=N$.
We will then have
\[
  \inf_{\theta\in\Theta} \tilde S^\theta_N
  =
  \E^{\tau}(S_N).
\]

\subsection{Infinite horizon}

One more disadvantage of Proposition~\ref{prop:exact}
is that it is only applicable to a finite horizon.
We can generalize it by allowing $N$ to be, e.g., a bounded stopping time,
but a natural question is whether it holds asymptotically at infinity
for the infinite horizon $N:=\infty$.
The next proposition is a step in this direction,
but it is very restrictive (as we will discuss momentarily).

\begin{proposition}\label{prop:asymptotic}
  Suppose that the parameter set $\Theta$ is finite
  and that different $P_{\theta}$ in the statistical model (our null hypothesis) $(P_{\theta})$
  are mutually singular.
  Let $(S^\theta)$ be a family of test martingales w.r.\ to the same filtration and $(P_{\theta})$,
  and let $\epsilon>0$ (be arbitrarily small).
  Then there exists a family of natural test martingales $(\tilde S^\theta)$
  such that
  \begin{equation}\label{eq:asymptotic}
    \liminf_{n\to\infty} \inf_{\theta} \tilde S^{\theta}_n
    \ge
    (1-\epsilon)
    \limsup_{n\to\infty} \inf_{\theta} S^{\theta}_n
  \end{equation}
  a.s.\ under any probability measure $P_{\theta}$ from the null hypothesis.
\end{proposition}

As in the case of Proposition~\ref{prop:exact},
Proposition~\ref{prop:asymptotic} says that, even when a natural test martingale $\tilde S^\theta$
falls below the original test martingale $S^\theta$,
it will eventually catch up
(or at least almost catch up, to within any $\epsilon$ on the relative scale).
The most restrictive condition in Proposition~\ref{prop:asymptotic}
is that $\Theta$ is finite
(although it can be as dense as we wish).

Another restrictive condition in Proposition~\ref{prop:asymptotic}
is that different $P_\theta$ are required to be mutually singular.
This condition often holds for interesting statistical models;
for example, in the IID case it follows from Kakutani's theorem \citep{Kakutani:1948}
that $P_\theta$ corresponding to different $\theta$
are either identical or mutually singular.
Moreover, we can often even identify $\theta$ in the limit almost surely
given a sequence observations generated from $P_\theta$
(formally, there exists a strongly consistent estimator for $\theta$).

In \eqref{eq:asymptotic} we have $\liminf$ and $\limsup$ instead of just $\lim$.
For the $\liminf$ it is not essential, and we can replace it by $\lim$,
meaning that the limit will exist almost surely
(although it can be $\infty$).
Having $\limsup$ is essential,
but let me discuss it at the end of the proof.

\begin{proof}[Proof of Proposition~\ref{prop:asymptotic}]
  By Doob's convergence theorem \citep[Corollary 7.4.3]{Shiryaev:2019},
  \begin{equation}\label{eq:S^theta}
    S^{\theta}_{\infty}
    :=
    \lim_{n\to\infty}
    S^\theta_n
  \end{equation}
  exists almost surely under $P_{\theta}$.
  Without loss of generality we assume that its $\E^\theta$-expectation is 1
  (its expectation is at most 1 by Fatou's lemma,
  and we can scale it up if the expectation is below 1).
  For each $\theta\in\Theta$ let us define the natural test martingale
  \begin{equation}\label{eq:S^theta-backward}
    \tilde S^{\theta}_n
    :=
    \E^{\theta}(S^{\theta}_{\infty}\mid\FFF_n),
    \quad
    n=0,1,2,\dots;
  \end{equation}
  remember that this process is a test martingale only under $P_{\theta}$.
  By L\'evy's theorem \citep[Theorem 7.4.3]{Shiryaev:2019}
  we have
  \[
    \tilde S^{\theta}_n
    \to
    S^{\theta}_{\infty}
  \]
  a.s.\ under $P_{\theta}$.
  Since this convergence holds only $P_\theta$-almost surely,
  we need to ``regularize'' $\tilde S^{\theta}$
  to ensure its desired behaviour under $P_{\theta'}$ for $\theta'\ne\theta$.

  For each pair $\theta,\theta'\in\Theta$ with $\theta\ne\theta'$,
  fix a natural test martingale $S^{\theta,\theta'}$ w.r.\ to $P_{\theta}$
  such that
  \[
    \liminf_{n\to\infty}
    S^{\theta,\theta'}
    =
    \infty
    \quad
    \text{$P_{\theta'}$-a.s.}
  \]
  Such a test martingale can be defined
  as the likelihood ratio $\d P_{\theta'} / \d P_{\theta}$
  if $P_{\theta'}$ is locally absolutely continuous w.r.\ to $P_{\theta}$
  (see \citealt[Theorem 7.6.2]{Shiryaev:2019})
  and as an obvious modification of the likelihood ratio otherwise.
  Now we can redefine
  \begin{equation}\label{eq:redefine}
    \tilde S^{\theta}_n
    :=
    (1-\epsilon)
    \tilde S^{\theta}_n
    +
    \frac{\epsilon}{\lvert\Theta\rvert-1}
    \sum_{\theta'\in\Theta\setminus\{\theta\}}
    S^{\theta,\theta'}_n
  \end{equation}
  (assuming, without loss of generality, that $\lvert\Theta\rvert>1$).

  Under $P_\theta$, we have, a.s.,
  \begin{equation}\label{eq:lim-1}
    \lim_{n\to\infty} \tilde S^{\theta}_n
    \ge
    (1-\epsilon)
    S^{\theta}_{\infty}
    =
    (1-\epsilon)
    \lim_{n\to\infty} S^{\theta}_n,
  \end{equation}
  and under $P_{\theta'}$, $\theta'\ne\theta$,
  we have, a.s.,
  \begin{equation}\label{eq:lim-2}
    \liminf_{n\to\infty} \tilde S^{\theta}_n
    \ge
    \frac{\epsilon}{\lvert\Theta\rvert-1}
    \liminf_{n\to\infty} S^{\theta,\theta'}_n
    =
    \infty.
  \end{equation}
  Combining \eqref{eq:lim-1} and \eqref{eq:lim-2},
  we obtain, almost surely under any element of the statistical model,
  \[
    \liminf_{n\to\infty}
    \min_{\theta\in\Theta}
    \tilde S^{\theta}_n
    =
    \min_{\theta\in\Theta}
    \liminf_{n\to\infty}
    \tilde S^{\theta}_n
    \ge
    (1-\epsilon)
    \min_{\theta\in\Theta}
    \limsup_{n\to\infty} S^{\theta}_n.
  \]

  Now we can discuss in detail
  the role of the $\liminf$ and $\limsup$ in \eqref{eq:asymptotic}.
  The $\liminf$ can be replaced by $\lim$
  since $\lim_{n\to\infty}\tilde S_n^\theta$ exists (and is finite, almost surely)
  under $P_\theta$ by Doob's convergence theorem
  and exists (and is $\infty$, almost surely) under $P_{\theta'}$ for $\theta'\ne\theta$
  because of the components $S_n^{\theta,\theta'}$ in \eqref{eq:redefine}.
  As for the $\limsup$, $\lim_{n\to\infty}S_n^\theta$ exists almost surely
  under $P_\theta$,
  but there are no constraints on $S_n^\theta$'s behaviour
  under $P_{\theta'}$ for $\theta'\ne\theta$;
  therefore, it is essential to have $\limsup$
  (unless we are willing to weaken \eqref{eq:asymptotic}).
\end{proof}

\subsection{Creating natural test martingales out of likelihood ratios}

So far in this section we were discussing creating natural test martingales out
of other test martingales (those w.r.\ to a reduced filtration).
But the process has a bottleneck:
first we define an e-variable
(i.e., a nonnegative random variable with expectation 1,
such as \eqref{eq:S^theta})
and then average it w.r.\ to a filtration (as in \eqref{eq:S^theta-backward}).
An easier option is to start directly from an e-variable over the first $N$ observations,
in the case of a finite horizon $N$.

When testing an online compression model (such as exchangeability),
this model serves as our null hypothesis.
We also fix an \emph{alternative hypothesis},
which, in the simplest case, is a probability measure $Q$ on the sample space.
(It can be the mixture of a Bayesian model,
as in \citealt[Sect.~9.2]{Vovk/etal:2022book}.)

What is really important for us is not $Q$ itself,
but a regular conditional probability generated from $Q$,
which in fact carries less information than $Q$ does.
(See, e.g., \citealt[Sect.~II.89]{Rogers/Williams:2000},
for a standard theorem about the existence of a suitable regular conditional probability.)
As a replacement for $Q$ in the context of testing an online compression model
with a summary statistic $t$,
we let $q$ to be a Markov kernel
mapping each $\sigma\in t(\mathbf{Z}^N)$
to a probability measure $q(\sigma)$ on the set $t_N^{-1}(\sigma)$.
To obtain a family of martingales $(S^\theta)$ from an alternative hypothesis,
in the case of a finite horizon $N$,
we can proceed as in Proposition~\ref{prop:exact},
namely we set
\begin{equation}\label{eq:batch_benchmark}
  S_N^{\theta}
  :=
  \frac{\dd q(\sigma_N)}{\dd P_N(\sigma_N)},
  \quad
  S_n^{\theta}
  :=
  \E^{\theta}
  \left(
    S_N^{\theta}
    \mid
    \FFF_n
  \right),
  \quad
  \theta\in\Theta,
  \quad
  n=0,\dots,N-1,
\end{equation}
where $P_N(\sigma_N)$ is the probability measure (on $t_N^{-1}(\sigma)$)
in the corresponding repetitive structure.
We will see some experimental results for the final value
\begin{equation}\label{eq:batch}
  \frac{\dd q(\sigma_N)}{\dd P_N(\sigma_N)}
\end{equation}
(which we call the \emph{batch benchmark})
of these test martingales in the next section.

In the case of an infinite horizon
and in the spirit of Proposition~\ref{prop:asymptotic},
we have an open problem.
Consider the sequence of the summaries $\sigma_1,\sigma_2,\dots$
generated from the alternative hypothesis $Q$.
For each of them (and for each value of $\theta$)
define the likelihood ratio martingale \eqref{eq:batch_benchmark}.
Under what conditions do these test martingales converge?
And if they do, can the limit be used for hypothesis testing?

\section{Illustration: the problem of change detection}
\label{sec:experiments}

In this section I will illustrate several points raised in the previous sections
using a simple setting of changepoint detection with a finite horizon.
Chapter~9 of \citet{Vovk/etal:2022book} shows numerous examples
where conformal testing
(usually implemented as the Bayes--Kelly, or \emph{BK}, conformal test martingale)
is very close to natural benchmarks.
In the setting of this section
the difference is deliberately made more pronounced.
Namely, the observation space is $\mathbf{Z}:=\{0,1\}$,
the null hypothesis is the randomness model
$(\mathbf{B}_{\theta}^{20}\mid\theta\in[0,1]$),
$\mathbf{B}_\theta$ being the Bernoulli distribution on $\{0,1\}$
with parameter $\theta$
(which is the probability of 1, $\mathbf{B}_{\theta}(\{1\})=\theta$),
the alternative hypothesis $Q$ is that
$N_0:=10$ observations are generated from the Bernoulli distribution
$\mathbf{B}_{\pi_0}=\mathbf{B}_{0.1}$
with parameter $\pi_0:=0.1$,
and another $N_1:=10$ observations are generated
from the Bernoulli distribution
$\mathbf{B}_{\pi_1}=\mathbf{B}_{0.9}$
with parameter $\pi_1:=0.9$.
(Notice that both $\pi_0$ and $\pi_1$ are probabilities of 1,
pre-change and post-change.)
All 20 observations are generated independently.
In this setting the time horizon is finite and very short, $N:=20$.

\begin{figure}[b]
  \begin{center}
    \includegraphics[width=0.48\textwidth]{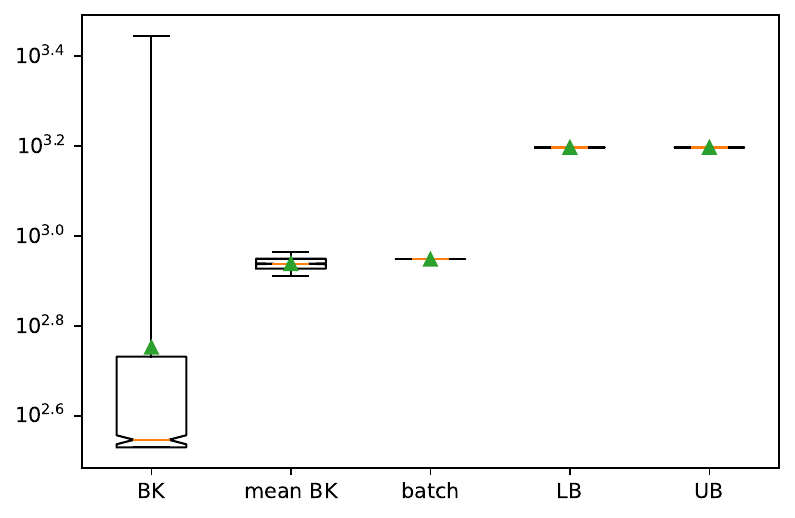}
    \hfill
    \includegraphics[width=0.48\textwidth]{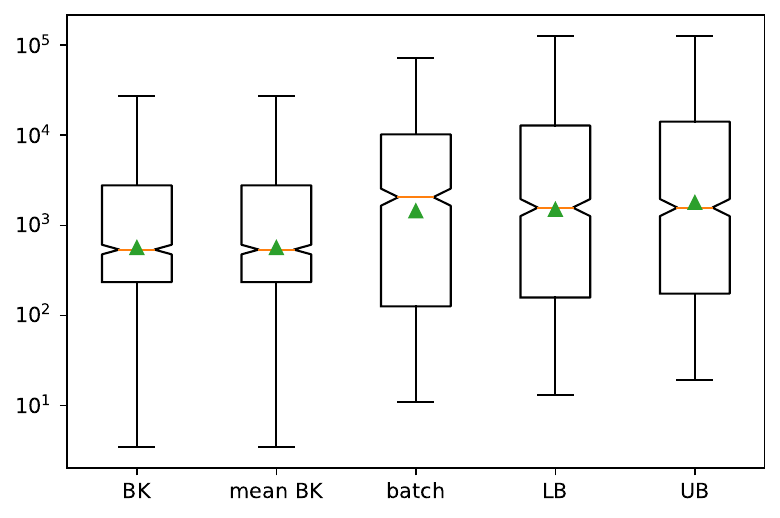}
  \end{center}
  \caption{Left panel:
    Five final values as described in text for a fixed dataset (for a changepoint detection problem).
    Right panel:
    Five final values as described in text for random datasets.}
  \label{fig:BK}
\end{figure}

We will only be interested in the final values of our martingales and related processes;
for some of these processes the intermediate values are easily computable,
but for others this is tricky (and requires further research).
The final values are shown in Fig.~\ref{fig:BK},
which will be explained in the rest of this section.

The boxplots in Fig.~\ref{fig:BK} represent results of $10^3$ independent simulations
of the final values (at time horizon 20) of five processes,
including the Bayes--Kelly martingale (BK).
Each boxplot shows the median as the horizontal orange line in the middle of a box,
with notches representing a confidence interval around the median,
the mean as a green triangle,
the interquartile range as a box,
and the $5\%$ and $95\%$ quantiles as whiskers.

The second process, ``mean BK'',
is an approximation to the expectation (cf.\ \eqref{eq:mean_BK})
of the BK test martingale $S_n$.
To compute its final value,
we compute the final values of $10^3$ independent simulations of the BK test martingale
and then average them.
If we average $S_n$ for each $n=0,\dots,N$,
the resulting process will no longer be a martingale,
as discussed in \citet[Sect.~9.3]{Vovk/etal:2022book}.
However, the final value will have expected value 1 under the null hypothesis
(and so will be a valid measure of evidence collected against the null hypothesis).

The last two processes in Fig.~\ref{fig:BK} are
the \emph{lower benchmark} (LB) and the \emph{upper benchmark} (UB).
The former is
\[
  \LB_n
  :=
  \inf_{\theta}
  \frac{Q(Z_1=z_1,\dots,Z_n=z_n)}{\mathbf{B}_\theta(\{z_1\})\dots\mathbf{B}_\theta(\{z_n\})},
\]
where $z_1,\dots,z_n$ are the realized values of the random observations $Z_1,\dots,Z_n$,
respectively,
and the latter is
\[
  \UB_n
  :=
  \frac{Q(Z_1=z_1,\dots,Z_n=z_n)}{\mathbf{B}_{0.5}(\{z_1\})\dots\mathbf{B}_{0.5}(\{z_n\})}.
\]
The lower benchmark is not a martingale under any $\mathbf{B}_{\theta}^{20}$,
but it is a valid measure of evidence against the null since for each $\theta$
it is dominated by the likelihood ratio
$Q(Z_1=z_1,\dots,Z_n=z_n)/(\mathbf{B}_\theta(\{z_1\})\dots\mathbf{B}_\theta(\{z_n\}))$.
On the other hand, the upper benchmark is only valid under $\mathbf{B}_{0.5}^{20}$
(which is, in a sense, the mid-point between the pre-change and post-change distributions),
and not valid under the other elements of the null hypothesis.

The middle process, the one labelled ``batch'',
is a new benchmark (which we called the ``batch benchmark'' earlier),
and we will discuss it at the end of this section.

The left panel of Fig.~\ref{fig:BK} shows
the final values of the BK martingale, mean BK martingale,
batch benchmark, and upper and lower benchmarks for a specific randomly generated dataset
(using our default seed 42 for the random number generator).
With a large probability,
the number of 1s in the dataset will be 10,
in which case the upper and lower benchmarks will in fact coincide,
as they do in Fig.~\ref{fig:BK}.

The final value of the mean BK martingale is higher
than that of the BK martingale in the left panel of Fig.~\ref{fig:BK},
and it is less volatile.
It is higher because averaging on the log scale is akin to taking maximum,
as we pointed out in \citet[Sect.~9.3]{Vovk/etal:2022book}.
It is clear that the genuine average (expectation) of the BK martingale
over the random numbers $\tau$ is even higher (with zero volatility),
but it is only marginally higher (as our other experiments show).

If the dataset is randomized, the difference is much less noticeable:
see the right panel of Fig.~\ref{fig:BK}.
In particular, the difference between BK and mean BK
is swamped by the variability due to the random choice of a dataset.
The three benchmarks, however, are still significantly higher
in mean and median.

Now let us spell out the batch benchmark \eqref{eq:batch},
shown in the middle boxplots of both panels of Fig.~\ref{fig:BK},
for this case.
Suppose the observed data sequence is $z_1,\dots,z_N$
and let
\[
  K := \sum_{n=1}^N z_n,
  \quad
  k_0 := \sum_{n=1}^{N_0} z_n,
  \quad
  k_1 := \sum_{n=N_0+1}^{N} z_n
\]
be the numbers of 1s among all observations,
among the pre-change observations,
and among the post-change observations,
respectively.
The probability of $z_1,\dots,z_N$ under the alternative is
\[
  \pi_0^{k_0} (1-\pi_0)^{N_0-k_0}
  \pi_1^{k_1} (1-\pi_1)^{N_1-k_1}
\]
and the number of data sequences leading to the same $k_0$ and $k_1$ is
\[
  \binom{N_0}{k_0}
  \binom{N_1}{k_1}.
\]
The exchangeability summary
(i.e., the summary under the exchangeability compression model)
of $z_1,\dots,z_N$ is $K$,
and so the conditional probability of $z_1,\dots,z_N$
given its exchangeability summary
under the alternative hypothesis is
\begin{multline*}
  \frac
  {
    \pi_0^{k_0} (1-\pi_0)^{N_0-k_0}
    \pi_1^{k_1} (1-\pi_1)^{N_1-k_1}
  }
  {
    \sum_{k=(K-N_0)^+}^{K\wedge N_1}
    \binom{N_0}{K-k}
    \binom{N_1}{k}
    \pi_0^{K-k} (1-\pi_0)^{N_0-K+k}
    \pi_1^{k} (1-\pi_1)^{N_1-k}
  }\\
  =
  \frac{1}
  {
    \sum_{k=(K-N_0)^+}^{K\wedge N_1}
    \binom{N_0}{K-k}
    \binom{N_1}{k}
    \left(
      \frac{(1-\pi_0)\pi_1}{\pi_0(1-\pi_1)}
    \right)^{k-k_1}
  }.
\end{multline*}
In this formula, $k$ is the analogue of $k_1$
for the generic element of $t_N^{-1}(\sigma_N)$
(where $t_N$ and $\sigma_N$ refer to the exchangeability model),
and $K-k$ is the analogue of $k_0$.
It is clear that $k$ ranges from $(K-N_0)^+$
(where $u^+:=\max(u,0)$)
and $K\wedge N_1$ (where $u\wedge v:=\min(u,v)$);
it is easy to check directly that $(K-N_0)^+\le K\wedge N_1$.
The conditional probability of $z_1,\dots,z_N$
given its exchangeability summary
under the null hypothesis is
\[
  1/\binom{N}{K},
\]
which gives the explicit expression
\[
  \frac{\binom{N}{K}}
  {
    \sum_{k=(K-N_0)^+}^{K\wedge N_1}
    \binom{N_0}{K-k}
    \binom{N_1}{k}
    \left(
      \frac{(1-\pi_0)\pi_1}{\pi_0(1-\pi_1)}
    \right)^{k-k_1}
  }
\]
for the batch benchmark \eqref{eq:batch} that we use in our experiments.

The right panel of Fig.~\ref{fig:BK} shows that the batch benchmark is competitive
with the lower and upper benchmarks.
It looks a promising option.
Its advantage over the upper benchmark is obvious:
it is valid under any power probability measure, not just under $\mathbf{B}_{0.5}^N$.
One advantage over the lower benchmark is that it is admissible for each parameter value $\theta$,
whereas the inadmissibility of the lower benchmark for some $\theta$ is obvious.

\section{Conclusion}
\label{sec:conclusion}

I have mentioned several directions of further research
in the previous sections,
but these are a few more:
\begin{itemize}
\item
  In Sect.~\ref{subsec:conformal} we saw that for the model $(\NNN^{\infty}_{\mu,1})$
  the element-wise tests are not fully adequate.
  It would be interesting to quantify this observation
  and to extend it to other online compression models.
\item
  In the examples of Sect.~\ref{sec:inadequacy}
  we used the fact (see \eqref{eq:p_2})
  that the reduced $\sigma$-algebras $\FFF'_2$ coincide
  for the pivotal and online compression methods
  in the case of the Gaussian model with variance 1.
  It can be shown that the other non-trivial reduced $\sigma$-algebras $\FFF'_n$,
  $n>2$,
  also coincide for that model.
  In general, however, the two methods may involve
  very different degrees of forgetting,
  which would be interesting to formalize and quantify for different models.
\item
  Can we apply Proposition~\ref{prop:exact} and \eqref{eq:mean_BK}
  (or their elaborations)
  to get explicit expressions for the natural modifications
  (i.e., modifications that are test martingales w.r.\ to the natural filtration)
  of the numerous conformal test martingales
  described in \citet[Part~III]{Vovk/etal:2022book}?
\item
  Relaxing the assumptions of Proposition~\ref{prop:asymptotic}
  (such as $\Theta$ being finite)
  or showing that it is impossible.
\item
  In Sect.~\ref{sec:experiments} we only studied
  the final values of various test martingales.
  Their intermediate values deserve to be studied
  both theoretically and experimentally.
\end{itemize}

A characteristic feature of conformal testing is that part of the data is forgotten
in the process of gambling against the null hypothesis (such as exchangeability).
On the other hand,
the same test martingale works against every probability measure in the null hypothesis.
We have seen that forgetting is essential,
even if our gambling strategy is allowed to depend on a probability measure
in the null hypothesis.

We have also seen that we can get rid of forgetting, but to a very limited extent.
It is not clear at all how the power and versatility of conformal testing
can be achieved without forgetting,
and it appears that, at least for the time being,
we should embrace the need for forgetting and live with it.

\subsection*{Acknowledgement}

Thanks to Wouter Koolen for an illuminating discussion of element-wise testing
as formalization of online testing during ISIPTA 2019 (3--6 July 2019).
Thanks to Aaditya Ramdas for another illuminating discussion
during a meeting in Amsterdam between the authors of \citet{Ramdas/etal:2023}
hosted by Peter Gr\"unwald (4--6 June 2022).
While the first discussion concentrated on advantages of element-wise testing,
the second one was all about their limitations.

Reviewers' comments on the conference version of this paper
led to numerous improvements in presentation,
and I appreciate their effort.

\end{document}